\documentclass[12pt]{amsart}
\usepackage[utf8]{inputenc}

\usepackage{mathrsfs}
\usepackage{amssymb}
\usepackage{bm}
\usepackage{graphicx}
\usepackage[centertags]{amsmath}
\usepackage{amsfonts}
\usepackage{amsthm}
\usepackage{enumerate}
\usepackage{tocvsec2}
\usepackage{xcolor}
\usepackage[margin=1in]{geometry}

\newtheorem{theorem}{Theorem}[section]
\newtheorem*{theorem*}{Theorem}

\newtheorem{rem}[theorem]{Remark}

\newtheorem{claim}{Claim}[section]

\theoremstyle{definition}
\newtheorem{definition}[theorem]{Definition}

\newcommand{\cts}{C(2^\nn)}

\newcommand{\nn}{\mathbb{N}}

\begin{document}

\title{On a generalization of Bourgain's tree index}

\author{Kevin Beanland}
\address{Department of Mathematics, Washington and Lee University, Lexington, VA 24450.}
\email{beanlandk@wlu.edu}

\author{Ryan M. Causey}
\address{Department of Mathematics, University of South Carolina, Columbia, SC 29208}
\email{causeyrm@mailbox.sc.edu}

\thanks{}

\thanks{2010 \textit{Mathematics Subject Classification}. Primary: }


\begin{abstract}
For a Banach space $X$, a sequence of Banach spaces $(Y_n)$, and a Banach space $Z$ with an unconditional basis, D. Alspach and B. Sari introduced a generalization of a Bourgain tree called a $(\oplus_n Y_n)_Z$-tree in $X$. These authors also  prove that any separable Banach space admitting a $(\oplus_n Y_n)_Z$-tree with order $\omega_1$ admits a subspace isomorphic to $(\oplus_n Y_n)_Z$. In this paper we give two new proofs of this result.
\end{abstract}

\maketitle

\section{Introduction}

In a recent work \cite{AS}, D. Alspach and B. Sari prove that every separable, elastic Banach space isomorphically contains $C[0,1]$, and is therefore universal for the class of separable Banach spaces. This remarkable result answers a question from an earlier deep work of W.B. Johnson and E. Odell \cite{JO}, in which they introduced the notion of elastic Banach space and use Bourgain's $\ell_\infty$-tree index \cite{Bourgain} to show that separable, elastic spaces isomorphically contain $c_0$. A critical step in Alspach and Sari's proof is to introduce the concept of a $(\oplus_n Y_n)_Z$-tree in $X$ and the corresponding $(\oplus_n Y_n)_Z$-index, which generalizes Bourgain's tree index.  We first recall the definition of a Bourgain tree, first given in \cite{Bourgain}.  Given a sequence $(y_i)$ having dense span in some Banach space $Y$ (usually a basis for $Y$), a Banach space $X$, and a positive constant $D$, we let $$T(X,D, (y_i))= \bigl\{(x_i)_{i=1}^n\subset X: (\forall (a_i)_{i=1}^n\subset \mathbb{F})(\|\sum_{i=1}^n a_ix_i\|\leqslant \|\sum_{i=1}^n a_iy_i\| \leqslant D\|\sum_{i=1}^n a_ix_i\|\bigr)\}.$$ 
Here, $\mathbb{F}$ denotes the scalar field.  It is easy to see that $Y$ $D$-embeds into $X$ if and only if there exists an infinite sequence $(x_i)_{i=1}^\infty\subset X$ such that for all $n\in \nn$, $(x_i)_{i=1}^n\in T(X,D,(y_i))$. We consider the following generalization.   

\begin{definition} Let $X$ be a Banach space, $(Y_n)$ a sequence of Banach spaces, $Z$ a Banach space with $1$-unconditional basis $(z_n)$, and $Y_Z$ denote the unconditional sum $(\oplus_n Y_n)_Z$. For constants $C, D\geqslant 1$, let $T(C,D,X)$ consist of all finite sequences of operators $(A_i)_{i=1}^n$ such that 

\begin{enumerate}[(i)]

\item for each $1\leqslant i\leqslant n$,  $A_i:Y_i\to X$ is an operator such that for all $y\in Y_i$, $\|y\| \leqslant \|A_iy\|\leqslant C\|y\|$, 

\item for each sequence $(y_i)_{i=1}^n$ such that for each $1\leqslant i\leqslant n$, $y_i\in Y_i$, $$\|\sum_{i=1}^n A_iy_i\|\leqslant \|\sum_{i=1}^n \|y_i\| z_i \|_Z \leqslant D\|\sum_{i=1}^n A_iy_i\|.$$  

\end{enumerate}
 We say a subset $T$ of $T(C,D,X)$ is a $(\oplus_n Y_n)_Z$-\emph{tree in} $X$ \emph{with constants} $C,D$ if for any $(A_i)_{i=1}^n\in T$ and each $1\leqslant m\leqslant n$, $(A_i)_{i=1}^m\in T$.        
\label{bigdef}
\end{definition}

Given a collection $T$ of finite sequences, we say $T$ is a \emph{tree} provided that any initial segment of a member of $T$ is also a member of $T$.  A subset $S$ of $T$ is a \emph{subtree} if it is also a tree.  Then $T$ is a $(\oplus_n Y_n)_Z$-tree in $X$ with constants $C,D$ if and only if it is a subtree of $T(C,D,X)$.  We note that the definition given above is not the one given by Alspach and Sari.  They considered subtrees of the tree $S(C,D,X)$ consisting of all finite sequences $(X_i, B_i)_{i=1}^n$ such that $X_i$ is a closed subspace of $X$, $B_i:X_i\to Y_i$ is an isomorphism,  and $(B_i^{-1})_{i=1}^n\in T(C,D,X)$. Note that the one-to-one correspondence between the sequences in $T(C,D,X)$ and the sequences of $S(C,D,X)$ given by $(A_i)_{i=1}^n \leftrightarrow (A_i(Y_i), A_i^{-1})_{i=1}^n$ also identifies subtrees of $T(C,D,X)$ with subtrees of $S(C,D,X)$.  Thus, for our purposes, it is sufficient to consider $T(C,D,X)$.


In \cite{AS}, the following generalization of a classical result for Bourgain trees is proved. 

\begin{theorem}\cite[Theorem 14]{AS} Suppose $X$ is separable, $(Y_n)$ is a sequence of separable Banach spaces, and $Z$ is a Banach space with a $1$-unconditional basis. Then for $D\geqslant 1$, if there exists $C>0$ such that there exists a $(\oplus_n Y_n)_Z$-tree with order $\omega_1$ in $X$ with constants $C,D$, $X$ has a subspace $D$-isomorphic to $Y_Z$.   
\label{question}
\end{theorem}

The main contribution of this note is to give two different proofs of the above result which might be considered more conceptual than the one found in \cite{AS}. Indeed, both proofs are straightforward applications of some of the descriptive set theoreitc machinery we present. 

The first proof uses a well-known result concerning trees satisfying a certain topological condition such that the members of the tree are sequences in a Polish space.  Prior to this work, the main obstruction towards using these results is the lack of an obvious Polish topology on collections of operators with varying domains.   We feel that this proof is interesting because it is a new application of a recently introduced coding of operators between separable Banach spaces. The second proof, on the other hand, is perhaps even simpler as it uses a pre-existing ordinal index on $X$. The second proof does not require the coding of operators, but instead follows from considering the relationship between the tree $T(C,D,X)$ and the tree $T(X,D,(y_i))$ given above, for an auspiciously chosen sequence $(y_i)$.   We begin with some definitions and preliminaries in the next section. In particular we recall the topology on the set of operators between separable Banach spaces and some well-known facts on Polish spaces and trees. The subsequent sections contain the proofs advertised above.

\section{Preliminaries}

Recall the definition of tree and subtree given in the introduction.  To avoid unnecessary technicalities, we do not define the order of a tree, but rather indicate the two properties of the order of a tree that we will need.  Given a tree $T$, the \emph{order} of $T$, denoted $o(T)$, is either a countable ordinal or the first uncountable ordinal.  Moreover, the order of a subtree of $T$ cannot exceed the order of the entire tree $T$.  If $S,T$ are any trees and if $f:S\to T$ is any function that takes proper extensions to proper extensions (recall that the members of $S$ and $T$ are finite sequences), then $o(S)\leqslant o(T)$.


If $\Lambda$ is a topological space, we say a tree $T$ on $\Lambda$ is \emph{closed} if for each $n\in\nn$, $T\cap \Lambda^n$ is closed in $\Lambda^n$ with its product topology. A separable, completely metrizable topological space is called a Polish space. For these spaces we have the following well-known theorem.

\begin{theorem}\cite{AGR} Let $T$ be a closed tree on a Polish space $P$.  Then $T$ is well-founded if and only if $o(T)<\omega_1$. 
\label{bdd}
\end{theorem}

Given a Polish space $P$, we let $F(P)$ denote the closed subsets of $P$.   It is a classical result of descriptive set theory that there exists a Polish topology on $F(P)$ \cite[page 75]{Ke}.    Another classical result is that there exists a sequence $d_n:F(P)\setminus\{\varnothing\}\to P$ of Borel functions (called \emph{Borel selectors}) such that for every $F\in F(P)\setminus \{\varnothing\}$, $\{d_n(F):n\in \nn\}$ is a dense subset of $F$ \cite{KRN}.   Fix a Polish topology on $F(\cts)$ and a sequence of Borel selectors $(d_n)$ as above, where $\cts$ denotes continuous functions on the Cantor set.   As is now standard convention, we let $\textbf{SB}$ denote all elements of $F(\cts)$ that are linear subspaces of $\cts$. The space $\textbf{SB}$ is Borel in $F(\cts)$ \cite[page 9]{DodosBook} with respect to our fixed topology on $F(\cts)$, and consequently there exists a Polish topology $\tau$ on $\textbf{SB}$ stronger than the topology inherited as a subspace of $F(\cts)$ \cite[page 75]{Ke}.  Therefore the functions $(d_n)$ are also Borel on $\textbf{SB}$. The space $\textbf{SB}$ is the coding of all separable Banach spaces and has been extensively studied in many contexts \cite{Bos, DodosBook}.   

We now endow $\textbf{SB}$ with a stronger Polish topology in the following way. Using \cite[page 82]{Ke}, we note that there is Polish topology $\tau'$ on $F(\cts)\setminus \{\varnothing\}$ stronger than the topology $\tau$ such that each $d_n$ is continuous with respect to $\tau'$.  In the sequel, the topology on $\textbf{SB}$ will always be $\tau'$, such that each Borel selector $d_n$ is continuous.

\subsection{The Polish Space $\mathfrak{L}$} As was explained in the introduction, one novelty of the current work is to isolate the proper Polish space on which to represent the $(\oplus_n Y_n)_Z$-trees. To do so, we recall a coding of the class of operators between separable Banach spaces first given in \cite{BF}.
First note that  $\textbf{SB}\times \textbf{SB}\times \cts^\nn$ is a Polish space when endowed with the product topology, where $\cts$ is endowed with its norm topology. Let $\mathfrak{L}$ consist of all triples $(X, Y, \hat{A})=(X, Y, (\hat{A}(n))_n)\in \textbf{SB}\times \textbf{SB}\times \cts^\nn$ such that for all $n\in \nn$, $\hat{A}(n)\in Y$, and such that there exists $k\in \nn$ such that for each rational linear combination $\sum_{i=1}^p q_i d_{n_i}(X)$ of $(d_n(X))$, we have $\|\sum_{i=1}^p q_i\hat{A}(n_i)\|\leqslant k\|\sum_{i=1}^p q_i d_{n_i}(X)\|$. In \cite{BCFW} it is proved that $\mathfrak{L}$ is a Borel subset of $\textbf{SB}\times \textbf{SB}\times \cts^\nn$, so that another appeal to \cite{Ke} yields the existence of a Polish topology on $\mathfrak{L}$ stronger than the subspace topology it inherits from the product topology on $\textbf{SB}\times \textbf{SB}\times \cts^\nn$. In the sequel, $\mathfrak{L}$ will always be endowed with this Polish topology.  Note that for each $n\in \nn$, $(X,Y, \hat{A})\mapsto d_n(X)$ is continuous with respect to this topology.  The purpose of this set is to act as a coding of all operators between separable Banach spaces.  To that end, for $X,Y \in \textbf{SB}$ and a bounded linear operator $A:X \to Y$, we have the unique tuple $(X,Y,\hat{A}) \in \mathfrak{L}$ where $\hat{A}(n):=Ad_n(X)$ for each $n\in \nn$; that is, $\hat{A}:\mathbb{N}\to \cts$ is defined as the image of the sequence $(d_n(X))\subset X$ under $A$. Conversely, to each tuple $(X,Y,\hat{A})$ we define the bounded linear operator $A:X\to Y$ by first defining $A$ on the dense subset $\{d_n(X):n\in \nn\}$ of $X$ by $Ad_n(X):=\hat{A}(n)$ and extending linearly and uniquely to all of $X$. It is proved in \cite[Claim 8.4]{BCFW} that the map $d_n(X)\mapsto \hat{A}(n)$ is well-defined and extends uniquely to a bounded, linear operator defined on all of $X$.

\section{Trees on $\mathfrak{L}$: The first proof}

Let $T$ be a $(\oplus_n Y_n)_Z$-tree in $X$ with constants $C,D$ having order $\omega_1$. Then $T(C,D,X)$ must have order $\omega_1$ as well, since it has a subtree with order $\omega_1$. In order to prove Theorem \ref{question}, it suffices to show that $T(C,D,X)$ is ill-founded. This will yield an infinite sequence $(A_i)_{i=1}^\infty$ all of whose finite initial segments lie in $T(C,D,X)$.  Then the operator $A:Y_Z\to X$ given by $A\sum y_i=\sum A_iy_i$ is easily seen to be a well-defined $D$-isomorphic embedding of $Y$ into $X$.   

Our method is to identify $T(C,D,X)$ with a closed tree $\mathcal{T}(C,D,X)$ on the Polish space $\mathfrak{L}$ in a way that also identifies subtrees of $T(C,D,X)$ with subtrees of $\mathcal{T}(C,D,X)$ having the same order, and that identifies well- (resp. ill-)founded trees with well-(resp. ill-)founded trees. Before defining $\mathcal{T}(C,D,X)$, we will suppose we have such an identification and finish the proof. Since the order of $T(C,D,X)$ is uncountable, so is the order of $\mathcal{T}(C,D,X)$.  Since $\mathcal{T}(C,D,X)$ is a closed tree on a Polish space, it is ill-founded by Theorem \ref{bdd}.  Therefore $T(C,D,X)$ is ill-founded as well.    

We proceed then to define the tree $\mathcal{T}(C,D,X)$ on $\mathfrak{L}$, establish the identification mentioned above,  and show that $\mathcal{T}(C,D,X)$ is closed. By embedding $X$ and $Y_Z$ isometrically into $C(2^\nn)$, we can assume $X, Y_Z\in \textbf{SB}$ and identify $Y_i$ with a subspace of $Y_Z$ in the natural way, so that $Y_i\in \textbf{SB}$.

Let $\mathcal{T}(C,D,X)$ consist of those sequences $(W_i, Z_i, \hat{A}_i)_{i=1}^n\in \mathfrak{L}^{<\nn}$ such that 

\begin{enumerate}[(i)]
\item for each $1\leqslant i\leqslant n$, $W_i=Y_i$, \item for each $1\leqslant i\leqslant n$, $Z_i=X$, \item for each $t\in \nn$ and $1\leqslant i\leqslant n$, $\|d_t(Y_i)\|\leqslant\|\hat{A}_i(t)\|\leqslant C\|d_t(Y_i)\|$,  \item for each sequence $(t_i)_{i=1}^n\subset \nn$, $$\|\sum_{i=1}^n \hat{A}_i(t_i)\|\leqslant \|\sum_{i=1}^n d_{t_i}(Y_i)\|\leqslant D\|\sum_{i=1}^n \hat{A}_i(t_i)\|.$$  
\end{enumerate}

The identification between $T(C,D,X)$ and $\mathcal{T}(C,D,X)$ is given by $(A_i)_{i=1}^n\leftrightarrow (Y_i, X, \hat{A}_i)_{i=1}^n$.   It is straightforward to check that $(A_i)_{i=1}^n\in T(C,D,X)$ if and only if $(Y_i, X, \hat{A}_i)_{i=1}^n\in \mathcal{T}(C,D,X)$.   It is also clear that this identification identifies subtrees with subtrees and preserves well- or ill-foundedness.   Note that this identification and its inverse maps proper extensions to proper extensions and so, by the remarks in the preliminaries, it preserves orders. 

We will prove that $\mathcal{T}(C,D,X)$ is a closed tree. For completeness we have included the details of the routine argument. Fix $n\in \nn$. We show that if $((W_i^j, Z_i^j, \hat{A}_i^j)_{i=1}^n)_{j=1}^\infty \subset \mathcal{T}(C,D,X)$ converges in $\mathfrak{L}$ to $(W_i, Z_i, \hat{A}_i)_{i=1}^n$, then $(W_i, Z_i, \hat{A}_i)_{i=1}^n\in \mathcal{T}(C,D,X)$. By definition of $\mathcal{T}(C,D,X)$, $W_i^j=Y_i$ and $Z^j_i=X$ for each $1\leqslant i\leqslant n$ and $j\in \nn$. By continuity of the coordinate projections $(X', Y', A')\mapsto X', Y'$ on $\mathfrak{L}$, for each $1\leqslant i\leqslant n$, $W_i=\lim_j W_i^j = \lim_j Y_i=Y_i$ and $Z_i=\lim_j Z^j_i=\lim_j X=X$.  This shows that (i) and (ii) are satisfied, and the limit $(W_i, Z_i, \hat{A}_i)_{i=1}^n$ is equal to $(Y_i, X, \hat{A}_i)_{i=1}^n$.  

By definition of $\mathcal{T}(C,D,X)$, for each $1\leqslant i\leqslant n$ and $j,t\in \nn$, $\|d_t(Y_i)\|\leqslant\|\hat{A}_i^j(t)\|\leqslant C\|d_t(Y_i)\|$ for each $j\in \nn$.  By continuity of $d_t$ and since $\hat{A}^j_i(t)\underset{j}{\to} \hat{A}_i(t)$, this same pair of inequalities is satisfied by the limit over $j$, whence (iii) is satisfied by $(Y_i, X, \hat{A}_i)_{i=1}^n$.  Similarly, for every sequence $(t_i)_{i=1}^n$ of natural numbers and every $j\in \nn$, $$\|\sum_{i=1}^n \hat{A}^j_i(t_i)\|\leqslant \|\sum_{i=1}^n d_{t_i}(Y_i)\|\leqslant D\|\sum_{i=1}^n \hat{A}^j_i(t_i)\|.$$ Using continuity of $d_{t_i}$ and the fact that $\hat{A}^j_i(t_i)\to \hat{A}_i(t_i)$, we deduce that the same pair of inequalities is satisfied after passing to the limit over $j$.   Therefore $(Y_i, X, \hat{A}_i)_{i=1}^n$ satisfies conditions (i)-(iv), and $\mathcal{T}(C,D,X)$ is closed.

\begin{rem}\upshape The reason we consider inverses of the original trees of Alspach and Sari is that, when the Borel selectors are continuous, the limit of a sequence of uniformly bounded isomorphisms with uniformly bounded inverses must be an isomorphic embedding with the same bounds on the norm and the norm of the inverse.  However, by considering a sequence of norm-one linear functionals converging weak* to zero, we can see that the limit of surjections need not be a surjection.

\end{rem}

\section{An alternative approach}

Fix a separable Banach space $Y$ and $(y_i)\subset Y$ a sequence with dense span in $Y$.  Given a separable Banach space $X$ and $D\geqslant 1$, as in the introduction, we may define the tree $$T(X, D, (y_i))=\Bigl\{(x_i)_{i=1}^n\in X^{<\nn}:( \forall (a_i)_{i=1}^n\subset \mathbb{F})( \|\sum_{i=1}^n a_ix_i \|\leqslant \|\sum_{i=1}^n a_iy_i\|\leqslant D\|\sum_{i=1}^n a_ix_i\|)\Bigr\}.$$  

This is clearly a closed tree.  Therefore if $X$ is separable, it follows that $T(X, D, (y_i))$ is well-founded if and only if $o(T(X, D, (y_i)))<\omega_1$.  It is also easy to see that $Y$ fails to $D$-embed into $X$ if and only if $o(T(X,D,(y_i)))<\omega_1$.   Indeed, by Theorem \ref{bdd}, $o(T(X,D, (y_i)))=\omega_1$ if and only if there exists a sequence $(x_i)_{i=1}^\infty$ all of whose finite initial segments lie in $T(X, D, (y_i))$.  But an infinite sequence $(x_i)_{i=1}^\infty\subset X$ has all of its finite initial segments in $T(X,D, (y_i))$  if and only if it is the image of the sequence $(y_i)$ under an isomorphic embedding of $Y$ into $X$ having norm at most $1$ the inverse of which has norm at most $D$.   

Suppose $Y=Y_Z$.  For each $n\in \nn$, let $S_n=(y^n_i)\subset Y_n\setminus \{0\}$ be a sequence dense in $Y_n$ (where, as usual, we consider $Y_n$ as a subspace of $Y$).  Let $(y_i)$ be some enumeration of $\cup_n S_n$ such that for each $k$, $(y_i)_{i=1}^k\subset \cup_{i=1}^k Y_i$.  Note that the span of $(y_i)$ is dense in $Y_Z$. In light of the previous paragraph, the Theorem \ref{question} is an immediate consequence of the following claim.  

\begin{claim} Let $X$ be a separable Banach space.  If there exists a $(\oplus_n Y_n)_Z$-tree in $X$ with constants $C,D$ having uncountable order, then $o(T(X, D, (y_i)))= \omega_1$, whence $Y_Z$ $D$-embeds into $X$.   

\end{claim}

\begin{proof} Let $T$ be an $(\oplus_n Y_n)_Z$ tree in $X$ with constants $C,D$ having order $\omega_1$.  As remarked in the preliminary section, it suffices to define a map $f:T\to T(X, D, (y_i))$ such that if $s,t\in T$ and $t$ is an proper extension of $s$, $f(t)$ is a proper extension of $f(s)$. From this it will follow that $\omega_1=o(T)\leqslant o(T(X, D, (y_i)))$.    Fix $t=(A_i)_{i=1}^n\in T$.  Choose $i_1, \ldots, i_n$ such that for each $1\leqslant j\leqslant n$, $y_j\in Y_{i_j}$. Note that for each $j$, $i_j$ is unique, since $y_j\neq 0$. Since $(y_i)_{i=1}^n\subset \cup_{i=1}^n Y_i$, $i_j\leqslant n$ for each $1\leqslant j\leqslant n$.  Let $f(t)= (A_{i_j}y_j)_{j=1}^n$. By the properties of $( A_i)_{i=1}^n$, $f(t)\in T(X, D, (y_i))$.   Since each $i_j$ is unique, if $1\leqslant m\leqslant n$ and if $s$ is the initial segment of $t$ having length $m$, $f(s)=(A_{i_j}y_j)_{j=1}^m$.  Therefore $f$ maps proper extensions to proper extensions.  
\end{proof}

\end{document}